\documentclass{amsart}

\usepackage{amsmath,amssymb}
\usepackage{hyperref,url,breakurl}
\usepackage{verbatimbox}
\usepackage{tikz}
\usepackage{algorithm}
\usepackage{algorithmicx}
\usepackage[noend]{algpseudocode}
\usepackage{booktabs}

\DeclareMathOperator{\Aut}{Aut}
\DeclareMathOperator{\lcm}{lcm}

\numberwithin{equation}{section}

\theoremstyle{plain}
\newtheorem{theorem}{Theorem}[section]
\newtheorem{proposition}[theorem]{Proposition}
\newtheorem{problem}[theorem]{Problem}

\theoremstyle{definition}
\newtheorem{example}[theorem]{Example}

\begin{document}

\title[Counting rank-3 lattices that have few coatoms]{Counting graded lattices of rank three\\that have few coatoms}

\author{Jukka Kohonen}

\address{Department of Computer Science, University of Helsinki,
PO Box 68, FI-00014 Helsinki, Finland}
\email{jukka\_kohonen@iki\_fi}

\maketitle

\begin{abstract}
  We consider the problem of computing $R(c,a)$, the number of
  unlabeled graded lattices of rank~$3$ that contain $c$~coatoms and
  $a$~atoms.  More specifically we do this when $c$~is fairly small,
  but $a$ may be large.  For this task, we describe a computational
  method that combines constructive listing of basic cases and tools
  from enumerative combinatorics.  With this method we compute the
  exact values of $R(c,a)$ for $c\le 9$ and $a\le 1000$.
  
  We also show that, for any fixed~$c$, there exists a quasipolynomial
  in~$a$ that matches with $R(c,a)$ for all $a$ above a small value.
  We explicitly determine these quasipolynomials for $c \le 7$, thus
  finding closed form expressions of $R(c,a)$ for~$c \le 7$.



\end{abstract}


\usetikzlibrary{matrix}
\usetikzlibrary{positioning}
\usetikzlibrary{calc}

\floatstyle{ruled}
\newfloat{program}{tbhp}{lop}
\floatname{program}{Program}

\newcommand{\Naturals}{\mathbb{N}}
\newcommand{\Integers}{\mathbb{Z}}
\newcommand{\Rationals}{\mathbb{Q}}
\newcommand{\Complex}{\mathbb{C}}
\newcommand{\ItGamma}{\mathit{\Gamma}}

\section{Introduction}
\label{sec:intro}

Let $R(c,a)$ denote the number of unlabeled graded lattices of rank~3
that contain $c$~coatoms and $a$~atoms, with $c,a \ge 1$.  One way of
determining the number is to actually generate the lattices.  But that
takes time at least linear in $R(c,a)$, and is thus limited to small
instances.

In this work we seek to compute $R(c,a)$ when $c$ is small, but $a$
may be large.  The converse case is handled by duality since
$R(c,a)=R(a,c)$.  We will not be content with approximate or
asymptotic results here: we want the exact numeric values and, if
possible, exact closed-form expressions of $R(c,a)$ for all $a \ge 1$.

This paper is structured as follows.  In Section~\ref{sec:related} we
briefly review related work and motivate the present study.
In~Section~\ref{sec:connection} we define a transformation of rank-3
lattices that will be helpful in counting.  In~Section~\ref{sec:three}
we manually derive $R(c,a)$ in closed form when $c \le 3$, and $a$~is
arbitrarily large.

For larger values of~$c$ we turn to computations.  In
Sections~\ref{sec:balls}---\ref{sec:implementation} we present a
computational method that uses nauty~\cite{nauty} to list the ways of
connecting $c$~coatoms, and then employs Redfield--P\'olya counting,
or the cycle index theorem from the theory of enumerative
combinatorics, to count distributions of atoms among the coatoms.
With this method, implemented using the computer algebra system
GAP~\cite{gap}, we have computed $R(c,a)$ for $c \le 9$ and $a \le
1000$.  In Section~\ref{sec:functional}, by doing exact polynomial
fits on residue classes, we derive $R(c,a)$ in closed form for $4 \le
c \le 7$ also.

\section{Related work and motivation}
\label{sec:related}

In the context of enumeration, a lattice is usually treated as a kind
of a directed graph, defined by its covering relation (or ``Hasse
graph'').  This places our counting problem in the realm of graph
enumeration (cf.~\cite{harary1973}).

Previous works that count lattices fall roughly into two categories.
For small lattices, exact counts are often determined by constructive
methods, by actually generating all lattices up to a certain size
\cite{gebhardt2016,heitzig2002,jipsen2015,kyuno1979}, or lattices in
some specific family~\cite{erne2002,jipsen2015}.  For large lattices
there are asymptotic upper and lower
bounds~\cite{erne2002,jipsen2015,kleitman1980,klotz1971}.  The present
work is somewhere in between: we seek exact counts of lattices that
are too large for constructive methods.

Enumeration of graded lattices seems to be a relatively uncharted
middle ground between the full generality of all lattices, and many
specific subfamilies.  Indeed, graded lattices encompass several
interesting subfamilies such as the distributive, modular, and
semimodular lattices.  There are several enumerative results both on
all lattices and on these subfamilies, but few results on graded
lattices.  Currently, the counts of graded lattices in OEIS
\cite{oeis} (sequence A278691) are based on simply generating them one
by one.

Efficient methods of counting graded lattices might be applicable to
some of their subfamilies as well.  In the other direction, graded
lattices may help the understanding of lattices in general.  An
example of such use is seen in the work of Kleitman and
Winston~\cite{kleitman1980}.  They first prove that the number of
rank-3 lattices of $n+2$ elements is at most $\beta^{n^{3/2} +
  o(n^{3/2})}$, where $\beta \approx 1.699408$.  From this they prove
an upper bound for all lattices, by using rank-3 lattices (``two level
lattices'') as building blocks.

It seems natural to begin the study of graded lattices from instances
that have few levels (low rank).  Graded lattices of rank~2 are
trivial: for any $n \ge 3$, there is only one rank-2 lattice of $n$
elements.  But rank-3 lattices are already quite nontrivial.  As shown
by Klotz and Lucht, their number grows faster than exponentially in
the number of elements~\cite{klotz1971}.

Of course, any algorithm that generates all lattices can be adapted to
rank-3 lattices.  However, the structure of rank-3 lattices makes them
amenable to enumerative methods that do not generate the lattices one
by one.  The present work combines constructive classification with
enumerative combinatorics.  We \emph{generate} a large number of basic
cases, and for each case we \emph{count} lattices without generating
them.  Both are done by computation, though the tools required are
quite different.

\section{Bicolored connection graphs}
\label{sec:connection}

A graded lattice of rank 3, or a \emph{rank-3 lattice} for short, has
elements on four levels: top, coatoms, atoms, and bottom.  From now on
we ignore the top and the bottom, since the structure of a rank-3
lattice is determined by its two central levels.  These levels form a
bicolored graph, whose two distinguished color classes are the coatoms
and the atoms.  Furthermore, there are no isolated elements, and any
two elements have at most one common neighbor
(cf.~\cite{kleitman1980}).

By the above characterization, rank-3 lattices can be counted up to
isomorphism with the following nauty~\cite{nauty} one-liner:
\begin{equation}
  \texttt{genbgL -d1 -Z1 -u \$c \$a}
  \label{eq:oneliner}
\end{equation}
This generates, up to isomorphism, all bicolored graphs that have $c$
vertices in the first color class and $a$ in the second, with minimum
degree one (\texttt{-d1}), such that any two vertices in the second
class have at most one common neighbor (\texttt{-Z1}).  The switch
\texttt{-u} specifies printing only the count and not the graphs.
Since this method actually generates the graphs one by one, it is
limited to small instances.  In this way the author has computed
$R(c,a)$ for $c+a \le 21$ (OEIS~\cite{oeis}, sequence A300260).

In order to count rank-3 lattices without generating them, it is
helpful to simplify their representation.  Let us divide the atoms
into two types, according to the number of coatoms by which they are
covered.  \emph{Connectors} are the atoms which are covered by at
least two coatoms, and the other atoms are called \emph{loners}.  In
lattice theory loners are known as meet-irreducible atoms, but we want
a shorter name here.  The \emph{connection graph} is the bicolored
graph spanned by the coatoms and the connectors.  Any rank-3 lattice
is uniquely represented by (1)~its connection graph and (2)~for each
coatom $u$, an integer $\ell(u)$ indicating how many loners it covers.
This is illustrated in Fig.~\ref{fig:connection}.

%
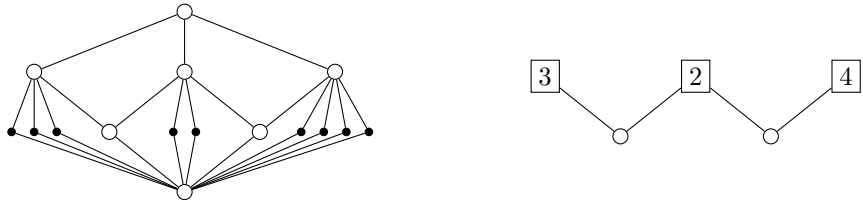
\begin{figure}[tb]
  \begin{minipage}[]{0.45\textwidth}
  \begin{tikzpicture}
    [extr/.style={circle,inner sep=0mm,minimum size=2mm,draw=black},
      coat/.style={circle,inner sep=0mm,minimum size=2mm,draw=black},
      conn/.style={circle,inner sep=0mm,minimum size=2mm,draw=black},
      lone/.style={circle,inner sep=0mm,minimum size=1mm,draw=black,fill=black}]
    \node[extr] (top) at (4,1.6) {};
    \node[extr] (bot) at (4,-.8) {};
    \node[coat] (u) at (2,.8) {};
    \node[coat] (v) at (4,.8) {};
    \node[coat] (w) at (6,.8) {};
    \node[conn] (a) at (3,0) {};
    \node[conn] (b) at (5,0) {};
    \node[lone] (u1) at (1.7,0) {};
    \node[lone] (u2) at (2.0,0) {};
    \node[lone] (u3) at (2.3,0) {};
    \node[lone] (v1) at (3.85,0) {};
    \node[lone] (v2) at (4.15,0) {};
    \node[lone] (w1) at (5.55,0) {};
    \node[lone] (w2) at (5.85,0) {};
    \node[lone] (w3) at (6.15,0) {};
    \node[lone] (w4) at (6.45,0) {};
    \draw[-] (top) -- (u);
    \draw[-] (top) -- (v);
    \draw[-] (top) -- (w);
    \draw[-] (u) -- (a);
    \draw[-] (v) -- (a);
    \draw[-] (v) -- (b);
    \draw[-] (w) -- (b);
    \draw[-] (u) -- (u1);
    \draw[-] (u) -- (u2);
    \draw[-] (u) -- (u3);
    \draw[-] (v) -- (v1);
    \draw[-] (v) -- (v2);
    \draw[-] (w) -- (w1);
    \draw[-] (w) -- (w2);
    \draw[-] (w) -- (w3);
    \draw[-] (w) -- (w4);
    \draw[-] (a) -- (bot);
    \draw[-] (b) -- (bot);
    \draw[-] (u1) -- (bot);
    \draw[-] (u2) -- (bot);
    \draw[-] (u3) -- (bot);
    \draw[-] (v1) -- (bot);
    \draw[-] (v2) -- (bot);
    \draw[-] (w1) -- (bot);
    \draw[-] (w2) -- (bot);
    \draw[-] (w3) -- (bot);
    \draw[-] (w4) -- (bot);    
  \end{tikzpicture}
  \end{minipage}
  \hfill
  \begin{minipage}[]{0.45\textwidth}
  \begin{tikzpicture}
    [coat/.style={rectangle,inner sep=1mm,minimum size=2mm,draw=black},
      conn/.style={circle,inner sep=0mm,minimum size=2mm,draw=black},
      lone/.style={circle,inner sep=0mm,minimum size=1mm,draw=black,fill=black}]
    \node[coat] (u) at (2,.8) {3};
    \node[coat] (v) at (4,.8) {2};
    \node[coat] (w) at (6,.8) {4};
    \node[conn] (a) at (3,0) {};
    \node[conn] (b) at (5,0) {};
    \draw[-] (u) -- (a);
    \draw[-] (v) -- (a);
    \draw[-] (v) -- (b);
    \draw[-] (w) -- (b);
  \end{tikzpicture}  
  \end{minipage}
  \caption{A rank-3 lattice (loners shown as small dots) and its
    connection graph (coatoms shown as boxes with numbers of loners).}
  \label{fig:connection}
\end{figure}

With any given value of $c$ there is only a finite collection of
connection graphs.  This can be seen as follows.  Let $r$ be the
number of connectors in a connection graph.  In a lattice, the covers
of two atoms are either disjoint, or intersect at a single element.
Thus we must have $r \le \binom{c}{2}$; the maximum is reached when
every pair of coatoms has its own connector.  Since the number of
connectors is bounded, the collection of connection graphs of
$c$~coatoms is finite.

For example, with $c=2$ there are two connection graphs, one with a
connector and one without.  With $c=3$, the number of connectors is $r
\in [0,3]$, and if $r=1$, the connector connects either two or three
coatoms.  So there are five connection graphs, which are shown in
Fig.~\ref{fig:c3}.

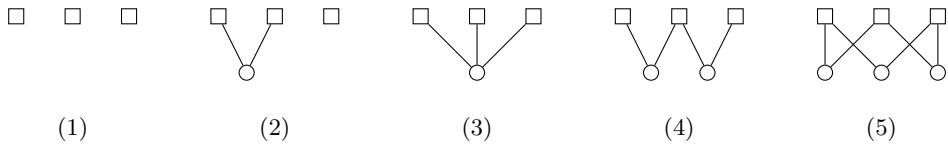
\begin{figure}[tb]
  \begin{minipage}[]{0.15\textwidth}
  \begin{tikzpicture}[scale=0.75,
    coat/.style={rectangle,inner sep=0mm,minimum size=2mm,draw=black},
      conn/.style={circle,inner sep=0mm,minimum size=2mm,draw=black},
      lone/.style={circle,inner sep=0mm,minimum size=1mm,draw=black,fill=black},
      caption/.style={rectangle,inner sep=0mm,minimum size=1mm,draw=none},
      none/.style={circle,inner sep=0mm,minimum size=2mm,draw=none,use as bounding box}]
    \node[coat] (u) at (0,1) {};
    \node[coat] (v) at (1,1) {};
    \node[coat] (w) at (2,1) {};
    \node[none] (a) at (0,0) {};
    \node[caption] (cap) at (1,-1) {\small(1)};
  \end{tikzpicture}  
  \label{graph:a}
  \end{minipage}
  \hfill
  \begin{minipage}[]{0.15\textwidth}
  \begin{tikzpicture}[scale=0.75,
    coat/.style={rectangle,inner sep=0mm,minimum size=2mm,draw=black},
      conn/.style={circle,inner sep=0mm,minimum size=2mm,draw=black},
      caption/.style={rectangle,inner sep=0mm,minimum size=1mm,draw=none},
      lone/.style={circle,inner sep=0mm,minimum size=1mm,draw=black,fill=black}]
    \node[coat] (u) at (0,1) {};
    \node[coat] (v) at (1,1) {};
    \node[coat] (w) at (2,1) {};
    \node[conn] (a) at (0.5,0) {};
    \node[caption] (cap) at (1,-1) {\small(2)};
    \draw[-] (u) -- (a);
    \draw[-] (v) -- (a);
  \end{tikzpicture}  
  \end{minipage}
  \hfill
  \begin{minipage}[]{0.15\textwidth}
  \begin{tikzpicture}[scale=0.75,
    coat/.style={rectangle,inner sep=0mm,minimum size=2mm,draw=black},
      conn/.style={circle,inner sep=0mm,minimum size=2mm,draw=black},
      caption/.style={rectangle,inner sep=0mm,minimum size=1mm,draw=none},
      lone/.style={circle,inner sep=0mm,minimum size=1mm,draw=black,fill=black}]
    \node[coat] (u) at (0,1) {};
    \node[coat] (v) at (1,1) {};
    \node[coat] (w) at (2,1) {};
    \node[conn] (a) at (1,0) {};
    \node[caption] (cap) at (1,-1) {\small(3)};
    \draw[-] (u) -- (a);
    \draw[-] (v) -- (a);
    \draw[-] (w) -- (a);
  \end{tikzpicture}  
  \end{minipage}
  \hfill
  \begin{minipage}[]{0.15\textwidth}
  \begin{tikzpicture}[scale=0.75,
    coat/.style={rectangle,inner sep=0mm,minimum size=2mm,draw=black},
      conn/.style={circle,inner sep=0mm,minimum size=2mm,draw=black},
      caption/.style={rectangle,inner sep=0mm,minimum size=1mm,draw=none},
      lone/.style={circle,inner sep=0mm,minimum size=1mm,draw=black,fill=black}]
    \node[coat] (u) at (0,1) {};
    \node[coat] (v) at (1,1) {};
    \node[coat] (w) at (2,1) {};
    \node[conn] (a) at (0.5,0) {};
    \node[conn] (b) at (1.5,0) {};
    \node[caption] (cap) at (1,-1) {\small(4)};
    \draw[-] (u) -- (a);
    \draw[-] (v) -- (a);
    \draw[-] (v) -- (b);
    \draw[-] (w) -- (b);
  \end{tikzpicture}  
  \end{minipage}
  \hfill
  \begin{minipage}[]{0.15\textwidth}
  \begin{tikzpicture}[scale=0.75,
    coat/.style={rectangle,inner sep=0mm,minimum size=2mm,draw=black},
      conn/.style={circle,inner sep=0mm,minimum size=2mm,draw=black},
      caption/.style={rectangle,inner sep=0mm,minimum size=1mm,draw=none},
      lone/.style={circle,inner sep=0mm,minimum size=1mm,draw=black,fill=black}]
    \node[coat] (u) at (0,1) {};
    \node[coat] (v) at (1,1) {};
    \node[coat] (w) at (2,1) {};
    \node[conn] (a) at (0,0) {};
    \node[conn] (b) at (1,0) {};
    \node[conn] (c) at (2,0) {};
    \node[caption] (cap) at (1,-1) {\small(5)};
    \draw[-] (u) -- (a);
    \draw[-] (v) -- (a);
    \draw[-] (u) -- (b);
    \draw[-] (w) -- (b);
    \draw[-] (v) -- (c);
    \draw[-] (w) -- (c);
  \end{tikzpicture}  
  \end{minipage}
  \caption{The five connection graphs of three coatoms.}
  \label{fig:c3}
\end{figure}

Now the unlabeled rank-3 lattices of $c$~coatoms and $a$~atoms can be
counted in two phases:
\begin{enumerate}
\begin{samepage}
\item \emph{List} all connection graphs of $c$ coatoms up to isomorphism.
\item For each connection graph, \emph{calculate} the number of ways
  to distribute $a-r$ loners among the $c$~coatoms.
\end{samepage}
\end{enumerate}

The first phase is straightforward using tools from
nauty~\cite{nauty}, although the list will be long if $c$ is large.
We use a variation of the previous one-liner~\eqref{eq:oneliner}.  Now
we allow elements in the the first color class (coatoms) to have zero
neighbors, since a coatom need not cover any connector; and we require
elements in the second color class (connectors) to have at least two
neighbors.  Thus we specify the minimum degrees in the two color
classes as zero and two (\texttt{-d0:2}).  We leave out \texttt{-u}
since we need the actual listing of the graphs.  We loop over all
possible values of $r$, from~$0$ to~$\binom{c}{2}=c(c-1)/2$.  The
following Bash script lists all connection graphs of $c$~coatoms up to
isomorphism:

\begin{verbbox}
let "rmax=c*(c-1)/2"
for ((r=0; r<=$rmax; r++))
do
    genbgL -d0:2 -Z1 $c $r
done
\end{verbbox}
\begin{equation}
  \theverbbox
  \label{eq:script}
\end{equation}

In the second phase, for each connection graph, we need to determine
how the remaining $a-r$ atoms (loners) can be located.  To this end,
we consider the problem of distributing $a-r$ identical balls into $c$
boxes.  Each such assignment corresponds to exactly one way to place
the remaining atoms and hence create a lattice.

However, a~couple of issues must be observed.  A simple thing is to
ensure that every coatom covers at least one atom.  If there are $s$
coatoms that do not cover a connector, we simply allocate one loner to
each, and distribute the remaining $n = a-r-s$ balls into $c$~boxes,
with empty boxes allowed.

Somewhat more complicated is to handle the symmetries of each
connection graph, so that we correctly count the isomorphism classes
of the lattices.  For example, in Fig.~\ref{fig:c3}, graphs (1), (3),
(5) have all coatoms in symmetric position, so we are putting balls
into three identical boxes.  But in graphs (2) and (4), two coatoms
are in symmetric position (two identical boxes) and the third is not
(a distinguished box).  Thus we must take care of the symmetry in our
counting work.  For $c$ at most three, we will settle this manually in
Section~\ref{sec:three}.  For the general case, we turn to
computational solutions.  This is described in
Section~\ref{sec:balls}.

\section{Lattices of at most three coatoms}
\label{sec:three}

With $c \le 3$ there are only a few connection graphs, and we can
treat them case-by-case.  Clearly $R(1,a)=1$.  Before going further we
recall some elementary results from the theory of enumerative
combinatorics.  Let $p_c(n)$ be the number of ways to distribute $n$
identical balls into $c$ identical boxes, empty boxes allowed.  For
all $n\ge 0$ we have~\cite[Theorem~6.4]{allenby2010}
\begin{equation}
  p_2(n) = \lfloor n/2 + 1 \rfloor,
  \label{eq:p2}
\end{equation}
\begin{equation}
  p_3(n) = \lfloor n^2/12 + n/2 + 1 \rfloor.
  \label{eq:p3}
\end{equation}
Also let $p_{2,1}(n)$ be the number of ways to distribute $n$
identical balls into two identical boxes and one distinguished box,
empty boxes allowed.  For all $n \ge 0$ we have
\begin{equation}
  p_{2,1}(n) = \sum_{i=0}^{n} p_2(i) = \lfloor n^2/4 + n + 1 \rfloor.
  \label{eq:p21}
\end{equation}
This may be recognized as the quarter-squares sequence
A002620~\cite{oeis}.  We adopt the convention that
$p_2(n)=p_3(n)=p_{2,1}(n)=0$ when $n<0$, and note that
Eqs.~\eqref{eq:p2}---\eqref{eq:p21} are then valid also for $p_2(-1)$,
$p_3(-2)$, $p_3(-1)$, and $p_{2,1}(-1)$.

\begin{theorem}
  For any $a\ge 1$, we have $R(2,a)=a$.
  \label{thm:c2}
\end{theorem}

\begin{proof}
  The lattice may or may not have a connector.  In either case, the
  two coatoms are in symmetric position.  If the lattice has a
  connector, we have $r=1$ and $s=0$, and the remaining $a-1$ atoms
  can be distributed in $p_2(a-1)$ ways.  If the lattice does not
  have a connector, we have $r=0$ and $s=2$, and the remaining $a-2$
  atoms can be distributed in $p_2(a-2)$ ways; this is zero if $a=1$.

  Adding up and simplifying, we get $R(2,a) = p_2(a-1)+p_2(a-2) = a$.
\end{proof}

\begin{theorem}
  For any $a\ge 1$, we have $R(3,a)=\lfloor (3/4)a^2 + (1/3)a + 1/4
  \rfloor$.
  \label{thm:c3}
\end{theorem}

\begin{proof}
  We consider the five possible connection graphs (recall
  Fig.~\ref{fig:c3}).

  Graph (1): No connectors, $r+s=3$.  The remaining $a-3$ atoms can be
  distributed among the three symmetric coatoms in $p_3(a-3)$ ways.

  Graph (2): One atom connects two coatoms, $r+s=2$.  The two
  connected coatoms are in symmetric position, so the remaining $a-2$
  atoms can be distributed in $p_{2,1}(a-2)$ ways.

  Graph (3): One atom connects three coatoms, $r+s=1$.  All coatoms are
  symmetric, so the remaining $a-1$ atoms can be distributed in
  $p_3(a-1)$ ways.

  Graph (4): Two connectors, $r+s=2$.  The two coatoms at the ends are
  symmetric, so the remaining $a-2$ atoms can be distributed in
  $p_{2,1}(a-2)$ ways.

  Graph (5): Three connectors, $r+s=3$.  All coatoms are symmetric, so
  the remaining $a-3$ atoms can be distributed in $p_3(a-3)$ ways.

  In all five cases, if $a < r+s$, then $p_3$ or $p_{2,1}$ has a
  negative argument, and a zero value by our convention.  Adding up we
  get
  \[
  R(3,a) = 2p_3(a-3) + p_3(a-1) + 2p_{2,1}(a-2).
  \]
  Substituting \eqref{eq:p3} and \eqref{eq:p21}, observing that they
  are valid for all $a \ge 1$ even if the arguments become negative,
  and simplifying, we obtain the stated result.
\end{proof}

\begin{figure}[tb]
  \begin{minipage}[]{0.15\textwidth}
  \begin{tikzpicture}[scale=0.4,
    coat/.style={rectangle,inner sep=0mm,minimum size=2mm,draw=black},
      conn/.style={circle,inner sep=0mm,minimum size=2mm,draw=black},
      lone/.style={circle,inner sep=0mm,minimum size=1mm,draw=black,fill=black},
      none/.style={circle,inner sep=0mm,minimum size=2mm,draw=none,use as bounding box}]
    \node[coat] (u) at (0,1) {};
    \node[coat] (v) at (1,1) {};
    \node[coat] (w) at (2,1) {};
    \node[coat] (z) at (3,1) {};
    \node[none] (a) at (0,0) {};
  \end{tikzpicture}
  \end{minipage}
  \hfill
  \begin{minipage}[]{0.15\textwidth}
  \begin{tikzpicture}[scale=0.4,
    coat/.style={rectangle,inner sep=0mm,minimum size=2mm,draw=black},
      conn/.style={circle,inner sep=0mm,minimum size=2mm,draw=black},
      lone/.style={circle,inner sep=0mm,minimum size=1mm,draw=black,fill=black}]
    \node[coat] (u) at (0,1) {};
    \node[coat] (v) at (1,1) {};
    \node[coat] (w) at (2,1) {};
    \node[coat] (z) at (3,1) {};
    \node[conn] (a) at (0.5,0) {};
    \draw[-] (u) -- (a);
    \draw[-] (v) -- (a);
  \end{tikzpicture}
  \end{minipage}
  \hfill
  \begin{minipage}[]{0.15\textwidth}
  \begin{tikzpicture}[scale=0.4,
    coat/.style={rectangle,inner sep=0mm,minimum size=2mm,draw=black},
      conn/.style={circle,inner sep=0mm,minimum size=2mm,draw=black},
      lone/.style={circle,inner sep=0mm,minimum size=1mm,draw=black,fill=black}]
    \node[coat] (u) at (0,1) {};
    \node[coat] (v) at (1,1) {};
    \node[coat] (w) at (2,1) {};
    \node[coat] (z) at (3,1) {};
    \node[conn] (a) at (1,0) {};
    \draw[-] (u) -- (a);
    \draw[-] (v) -- (a);
    \draw[-] (w) -- (a);
  \end{tikzpicture}
  \end{minipage}
  \hfill
  \begin{minipage}[]{0.15\textwidth}
  \begin{tikzpicture}[scale=0.4,
    coat/.style={rectangle,inner sep=0mm,minimum size=2mm,draw=black},
      conn/.style={circle,inner sep=0mm,minimum size=2mm,draw=black},
      lone/.style={circle,inner sep=0mm,minimum size=1mm,draw=black,fill=black}]
    \node[coat] (u) at (0,1) {};
    \node[coat] (v) at (1,1) {};
    \node[coat] (w) at (2,1) {};
    \node[coat] (z) at (3,1) {};
    \node[conn] (a) at (1.5,0) {};
    \draw[-] (u) -- (a);
    \draw[-] (v) -- (a);
    \draw[-] (w) -- (a);
    \draw[-] (z) -- (a);
  \end{tikzpicture}
  \end{minipage}
  \hfill
  \begin{minipage}[]{0.15\textwidth}
  \begin{tikzpicture}[scale=0.4,
    coat/.style={rectangle,inner sep=0mm,minimum size=2mm,draw=black},
      conn/.style={circle,inner sep=0mm,minimum size=2mm,draw=black},
      lone/.style={circle,inner sep=0mm,minimum size=1mm,draw=black,fill=black}]
    \node[coat] (u) at (0,1) {};
    \node[coat] (v) at (1,1) {};
    \node[coat] (w) at (2,1) {};
    \node[coat] (z) at (3,1) {};
    \node[conn] (a) at (0.5,0) {};
    \node[conn] (b) at (1.5,0) {};
    \draw[-] (u) -- (a);
    \draw[-] (v) -- (a);
    \draw[-] (v) -- (b);
    \draw[-] (w) -- (b);
  \end{tikzpicture}
  \end{minipage}
  \hfill
  \begin{minipage}[]{0.15\textwidth}
  \begin{tikzpicture}[scale=0.4,
    coat/.style={rectangle,inner sep=0mm,minimum size=2mm,draw=black},
      conn/.style={circle,inner sep=0mm,minimum size=2mm,draw=black},
      lone/.style={circle,inner sep=0mm,minimum size=1mm,draw=black,fill=black}]
    \node[coat] (u) at (0,1) {};
    \node[coat] (v) at (1,1) {};
    \node[coat] (w) at (2,1) {};
    \node[coat] (z) at (3,1) {};
    \node[conn] (a) at (0.5,0) {};
    \node[conn] (b) at (2.5,0) {};
    \draw[-] (u) -- (a);
    \draw[-] (v) -- (a);
    \draw[-] (w) -- (b);
    \draw[-] (z) -- (b);
  \end{tikzpicture}
  \end{minipage}
  \\[0.5cm]
  \begin{minipage}[]{0.15\textwidth}
  \centering
  \begin{tikzpicture}[scale=0.4,
    coat/.style={rectangle,inner sep=0mm,minimum size=2mm,draw=black},
      conn/.style={circle,inner sep=0mm,minimum size=2mm,draw=black},
      lone/.style={circle,inner sep=0mm,minimum size=1mm,draw=black,fill=black}]
    \node[coat] (u) at (0,1) {};
    \node[coat] (v) at (1,1) {};
    \node[coat] (w) at (2,1) {};
    \node[coat] (z) at (3,1) {};
    \node[conn] (a) at (1,0) {};
    \node[conn] (b) at (2.5,0) {};
    \draw[-] (u) -- (a);
    \draw[-] (v) -- (a);
    \draw[-] (w) -- (a);
    \draw[-] (w) -- (b);
    \draw[-] (z) -- (b);
  \end{tikzpicture}
  \end{minipage}
  \hfill
  \begin{minipage}[]{0.15\textwidth}
    \centering
  \begin{tikzpicture}[scale=0.4,
    coat/.style={rectangle,inner sep=0mm,minimum size=2mm,draw=black},
      conn/.style={circle,inner sep=0mm,minimum size=2mm,draw=black},
      lone/.style={circle,inner sep=0mm,minimum size=1mm,draw=black,fill=black}]
    \node[coat] (u) at (0,1) {};
    \node[coat] (v) at (1,1) {};
    \node[coat] (w) at (2,1) {};
    \node[coat] (z) at (3,1) {};
    \node[conn] (a) at (0,0) {};
    \node[conn] (b) at (1,0) {};
    \node[conn] (c) at (2,0) {};
    \draw[-] (u) -- (a);
    \draw[-] (v) -- (a);
    \draw[-] (u) -- (b);
    \draw[-] (w) -- (b);
    \draw[-] (v) -- (c);
    \draw[-] (w) -- (c);
  \end{tikzpicture}
  \end{minipage}
  \hfill
  \begin{minipage}[]{0.15\textwidth}
    \centering
  \begin{tikzpicture}[scale=0.4,
    coat/.style={rectangle,inner sep=0mm,minimum size=2mm,draw=black},
      conn/.style={circle,inner sep=0mm,minimum size=2mm,draw=black},
      lone/.style={circle,inner sep=0mm,minimum size=1mm,draw=black,fill=black}]
    \node[coat] (u) at (0,1) {};
    \node[coat] (v) at (1,1) {};
    \node[coat] (w) at (2,1) {};
    \node[coat] (z) at (3,1) {};
    \node[conn] (a) at (0,0) {};
    \node[conn] (b) at (1,0) {};
    \node[conn] (c) at (2,0) {};
    \draw[-] (u) -- (a);
    \draw[-] (v) -- (a);
    \draw[-] (u) -- (b);
    \draw[-] (w) -- (b);
    \draw[-] (u) -- (c);
    \draw[-] (z) -- (c);
  \end{tikzpicture}
  \end{minipage}
  \hfill
  \begin{minipage}[]{0.20\textwidth}
  \centering
  \begin{tikzpicture}[scale=0.4,
    coat/.style={rectangle,inner sep=0mm,minimum size=2mm,draw=black},
      conn/.style={circle,inner sep=0mm,minimum size=2mm,draw=black},
      lone/.style={circle,inner sep=0mm,minimum size=1mm,draw=black,fill=black}]
    \node[coat] (u) at (0,1) {};
    \node[coat] (v) at (1,1) {};
    \node[coat] (w) at (2,1) {};
    \node[coat] (z) at (3,1) {};
    \node[conn] (a) at (0.5,0) {};
    \node[conn] (b) at (1.5,0) {};
    \node[conn] (c) at (2.5,0) {};
    \draw[-] (u) -- (a);
    \draw[-] (v) -- (a);
    \draw[-] (v) -- (b);
    \draw[-] (w) -- (b);
    \draw[-] (w) -- (c);
    \draw[-] (z) -- (c);
  \end{tikzpicture}
  \end{minipage}
  \hfill
  \begin{minipage}[]{0.22\textwidth}
    \centering
  \begin{tikzpicture}[scale=0.4,
    coat/.style={rectangle,inner sep=0mm,minimum size=2mm,draw=black},
      conn/.style={circle,inner sep=0mm,minimum size=2mm,draw=black},
      lone/.style={circle,inner sep=0mm,minimum size=1mm,draw=black,fill=black}]
    \node[coat] (u) at (0,1) {};
    \node[coat] (v) at (1,1) {};
    \node[coat] (w) at (2,1) {};
    \node[coat] (z) at (3,1) {};
    \node[conn] (a) at (0,0) {};
    \node[conn] (b) at (1,0) {};
    \node[conn] (c) at (2,0) {};
    \draw[-] (u) -- (a);
    \draw[-] (v) -- (a);
    \draw[-] (u) -- (b);
    \draw[-] (w) -- (b);
    \draw[-] (v) -- (c);
    \draw[-] (w) -- (c);
    \draw[-] (z) -- (c);
  \end{tikzpicture}
  \end{minipage}
  \\[0.5cm]
  \begin{minipage}[]{0.15\textwidth}
  \centering
  \begin{tikzpicture}[scale=0.4,
    coat/.style={rectangle,inner sep=0mm,minimum size=2mm,draw=black},
      conn/.style={circle,inner sep=0mm,minimum size=2mm,draw=black},
      lone/.style={circle,inner sep=0mm,minimum size=1mm,draw=black,fill=black}]
    \node[coat] (u) at (0,1) {};
    \node[coat] (v) at (1,1) {};
    \node[coat] (w) at (2,1) {};
    \node[coat] (z) at (3,1) {};
    \node[conn] (a) at (0,0) {};
    \node[conn] (b) at (1,0) {};
    \node[conn] (c) at (2,0) {};
    \node[conn] (d) at (3,0) {};
    \draw[-] (u) -- (a);
    \draw[-] (v) -- (a);
    \draw[-] (u) -- (b);
    \draw[-] (w) -- (b);
    \draw[-] (v) -- (c);
    \draw[-] (w) -- (c);
    \draw[-] (w) -- (d);
    \draw[-] (z) -- (d);
  \end{tikzpicture}
  \end{minipage}
  \hfill
  \begin{minipage}[]{0.15\textwidth}
  \centering
  \begin{tikzpicture}[scale=0.4,
    coat/.style={rectangle,inner sep=0mm,minimum size=2mm,draw=black},
      conn/.style={circle,inner sep=0mm,minimum size=2mm,draw=black},
      lone/.style={circle,inner sep=0mm,minimum size=1mm,draw=black,fill=black}]
    \node[coat] (u) at (0,1) {};
    \node[coat] (v) at (1,1) {};
    \node[coat] (w) at (2,1) {};
    \node[coat] (z) at (3,1) {};
    \node[conn] (a) at (0,0) {};
    \node[conn] (b) at (1,0) {};
    \node[conn] (c) at (2,0) {};
    \node[conn] (d) at (3,0) {};
    \draw[-] (u) -- (a);
    \draw[-] (v) -- (a);
    \draw[-] (w) -- (a);
    \draw[-] (u) -- (b);
    \draw[-] (z) -- (b);
    \draw[-] (v) -- (c);
    \draw[-] (z) -- (c);
    \draw[-] (w) -- (d);
    \draw[-] (z) -- (d);
  \end{tikzpicture}
  \end{minipage}
  \hfill
  \begin{minipage}[]{0.15\textwidth}
  \centering
  \begin{tikzpicture}[scale=0.4,
    coat/.style={rectangle,inner sep=0mm,minimum size=2mm,draw=black},
      conn/.style={circle,inner sep=0mm,minimum size=2mm,draw=black},
      lone/.style={circle,inner sep=0mm,minimum size=1mm,draw=black,fill=black}]
    \node[coat] (u) at (0,1) {};
    \node[coat] (v) at (1,1) {};
    \node[coat] (w) at (2,1) {};
    \node[coat] (z) at (3,1) {};
    \node[conn] (a) at (0,0) {};
    \node[conn] (b) at (1,0) {};
    \node[conn] (c) at (2,0) {};
    \node[conn] (d) at (3,0) {};
    \draw[-] (u) -- (a);
    \draw[-] (v) -- (a);
    \draw[-] (v) -- (b);
    \draw[-] (w) -- (b);
    \draw[-] (w) -- (c);
    \draw[-] (z) -- (c);
    \draw[-] (z) -- (d);
    \draw[-] (u) -- (d);
  \end{tikzpicture}
  \end{minipage}
  \hfill
  \begin{minipage}[]{0.20\textwidth}
  \centering
  \begin{tikzpicture}[scale=0.4,
    coat/.style={rectangle,inner sep=0mm,minimum size=2mm,draw=black},
      conn/.style={circle,inner sep=0mm,minimum size=2mm,draw=black},
      lone/.style={circle,inner sep=0mm,minimum size=1mm,draw=black,fill=black}]
    \node[coat] (u) at (0,1) {};
    \node[coat] (v) at (1,1) {};
    \node[coat] (w) at (2,1) {};
    \node[coat] (z) at (3,1) {};
    \node[conn] (a) at (-0.5,0) {};
    \node[conn] (b) at (0.5,0) {};
    \node[conn] (c) at (1.5,0) {};
    \node[conn] (d) at (2.5,0) {};
    \node[conn] (e) at (3.5,0) {};
    \draw[-] (u) -- (a);
    \draw[-] (u) -- (b);
    \draw[-] (u) -- (c);
    \draw[-] (v) -- (a);
    \draw[-] (v) -- (d);
    \draw[-] (w) -- (b);
    \draw[-] (w) -- (e);
    \draw[-] (z) -- (c);
    \draw[-] (z) -- (d);
    \draw[-] (z) -- (e);
  \end{tikzpicture}
  \end{minipage}
  \hfill
\begin{minipage}[]{0.22\textwidth}
  \centering
  \begin{tikzpicture}[scale=0.4,
    coat/.style={rectangle,inner sep=0mm,minimum size=2mm,draw=black},
      conn/.style={circle,inner sep=0mm,minimum size=2mm,draw=black},
      lone/.style={circle,inner sep=0mm,minimum size=1mm,draw=black,fill=black}]
    \node[coat] (u) at (0,1) {};
    \node[coat] (v) at (1,1) {};
    \node[coat] (w) at (2,1) {};
    \node[coat] (z) at (3,1) {};
    \node[conn] (a) at (-1,0) {};
    \node[conn] (b) at (0,0) {};
    \node[conn] (c) at (1,0) {};
    \node[conn] (d) at (2,0) {};
    \node[conn] (e) at (3,0) {};
    \node[conn] (f) at (4,0) {};
    \draw[-] (u) -- (a);
    \draw[-] (v) -- (a);
    \draw[-] (u) -- (b);
    \draw[-] (w) -- (b);
    \draw[-] (u) -- (c);
    \draw[-] (z) -- (c);
    \draw[-] (v) -- (d);
    \draw[-] (w) -- (d);
    \draw[-] (v) -- (e);
    \draw[-] (z) -- (e);
    \draw[-] (w) -- (f);
    \draw[-] (z) -- (f);
  \end{tikzpicture}
  \end{minipage}
  \caption{The sixteen connection graphs of four coatoms.}
  \label{fig:c4}
\end{figure}
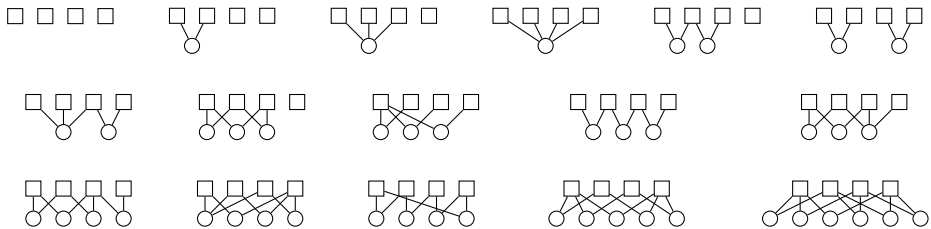

For $c=4$ one could continue in the same manner.  But there are
sixteen connection graphs (see Fig.~\ref{fig:c4}), exhibiting various
kinds of symmetry, and manual case-by-case analysis becomes tedious
and error-prone.  We now turn to a computational method that analyzes
the symmetries of the connection graphs, and also provides a numerical
solution to our counting problem.  In fact, the method will also help
us to derive closed form expressions for $R(4,a)$, $R(5,a)$, $R(6,a)$,
and $R(7,a)$.

\section{Balls into boxes with symmetry}
\label{sec:balls}

We now address the second phase of our task: given a bicolored
connection graph~$\ItGamma$, and an integer~$n$, count the ways of
distributing $n$~identical balls (the loners) among the vertices in
the first color class (the coatoms).  By saying ``bicolored'' we
distinguish the connectors from the coatoms.  Otherwise, for example,
in graph~5 of Fig.~\ref{fig:c3} there would be an automorphism that
swaps the classes.  Alternatively, we can orient the edges from atoms
to coatoms.

The count is affected by the automorphism group $\Aut(\ItGamma)$, or
in fact by its action \emph{on the coatoms}, since that is where we
place the balls.  For example, although in Fig.~\ref{fig:c3} the
graphs (1), (3), (5) are different, they have the same symmetry
between the coatoms.  For us they have the same kind of solution,
namely $p_3(\cdot)$ as we saw in the previous section.

Let us state the ball distribution problem in a more general form that
does not involve lattices or graphs.

\begin{problem}
  Given a group $G$ that acts on a set of boxes $\{1,2,\ldots,c\}$,
  count the ways of distributing $n$ identical balls in the boxes,
  when empty boxes are allowed, and two distributions related by an
  element of $G$ are considered equal.
  \label{prob:boxes}
\end{problem}

Some special cases of Problem~\ref{prob:boxes} are well-known.  For
example, if $G$ is the trivial group, then we have the problem with
$c$~distinguished boxes, and the count is $\binom{n+c-1}{c-1}$ by the
stars and bars argument.  If $G$ is the symmetric group, then the
boxes are in effect identical, and the solution involves partitions of
the integer~$n$ into at most $c$ nonnegative parts.

Despite its seemingly elementary nature, Problem~\ref{prob:boxes} is
not commonly found in textbooks.  However, in a thesis by
Lison\v{e}k~\cite{lisonek1994}, Section~4.3 is dedicated to the
problem of counting such ball distributions.  Lison\v{e}k shows that
the distributions (which he calls $G$-partitions) can be counted by an
application of the cycle index theorem from the Redfield--P\'olya
counting theory.  For a general introduction to the theory see for
example~\cite{cameron2017}.  Lison\v{e}k further shows that the count
has a representation as a quasipolynomial; we will return to this
topic later.  Beware that Lison\v{e}k's symbols are the reverse of
ours, as he puts $c$~balls into $n$~boxes.  Some other mentions of the
problem include~\cite[Chapter~14]{beekman2017} and \cite{cameron2018}.

For completeness of exposition, we describe a solution to
Problem~\ref{prob:boxes} here, without any claim of originality.  In
our application, we have $G=\Aut(\ItGamma)$, and we consider its
action on the $c$~coatoms.

For each group element $g \in G$, let $m_j(g)$ be the number of cycles
of length~$j$ in the disjoint cycle representation of $g$, and define
the cycle index monomial in $c$ indeterminates $t_1,\ldots,t_c$ as
follows:
\begin{equation*}
  z_g(t_1,\ldots,t_c) = t_1^{m_1(g)} t_2^{m_2(g)} \cdots t_c^{m_c(g)}.
\end{equation*}
The cycle index of $G$ (in its action on the boxes) is the average of
the cycle index monomials,
\begin{equation*}
  Z_G(t_1,\ldots,t_c) = \frac{1}{|G|} \sum_{g \in G} z_g(t_1,\ldots,t_c).
\end{equation*}

A distribution of balls into $c$ boxes (loners to coatoms in our
application) is described by a function $\ell : \{1,2,\ldots,c\} \to
\Naturals$, with $\ell(i)$ recording how many balls are in the $i$th
box.  Using terminology from counting theory, the nonnegative integers
$\ell(i)$ are \emph{figures} attached to each box.  Each figure has a
weight, which is here simply the integer itself.  The weight of a
function $\ell$ is $\sum_{i=1}^c \ell(i)$.  This is the total number
of balls in the distribution, so we want to count functions of
weight~$n$.

We further define a figure-counting series
\begin{equation}
  A(x) = 1+x+x^2+\ldots = 1/(1-x),
  \label{eq:figurecounting}
\end{equation}
which says that there exists one figure of each nonnegative weight
(the integer itself).  We want the function-counting series
\begin{equation*}
  B(x) = \sum_{n \ge 0} b_n x^n,
\end{equation*}
whose coefficient $b_n$ counts, up to symmetry, the functions $\ell$
of weight $n$.  By the cycle index theorem~\cite[Theorem
  7.3]{cameron2017} the function-counting series is
\begin{equation}
  B(x) = Z_G(A(x),A(x^2),\ldots,A(x^c)).
  \label{eq:functioncounting}
\end{equation}
Thus, as soon as we know the cycle index $Z_G$,
Eq.~\eqref{eq:functioncounting} gives a series whose coefficients are
the counts we desire (for all $n \ge 0$).  This amounts to a very
short piece of GAP code (Program~\ref{program:balls}).  As an
implementation detail, if we are content with a finite sequence of
results, we may replace the series $A(x)$ with the polynomial
$1+x+\ldots+x^n$ without affecting the coefficients of $B(x)$ up to
$x^n$.

\begin{program}[tb]
\begin{verbatim}
GroupBalls := function(ZG, c, n)
    local i,x,A,Axi,B;
    # Cycle index may use indeterminates 1..c, take next for x.
    x := Indeterminate(Rationals, c+1);
    # Figure-counting polynomial up to n, and substitute x^i.
    A := (1-x^(n+1)) / (1-x);
    Axi := [];
    for i in [1..c] do
        Axi[i] := Value(A, [x], [x^i]);
    od;
    # Function-counting polynomial
    B := Value(ZG, [1..c], Axi);
    # Pick the coefficients for x^0, ..., x^n
    return CoefficientsOfUnivariatePolynomial(B){[1..n+1]};
end;
\end{verbatim}
\caption{Count distributions of $0,\ldots,n$ balls into $c$ boxes,
  whose permutation group has cycle index $Z_G$.  Counts returned as a
  vector.}
\label{program:balls}
\end{program}

\begin{example}
  \label{ex:simple}
  \begin{samepage}
  Consider the following connection graph of four coatoms
  $a,b,c,d$.
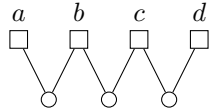
  \begin{center}
  \begin{minipage}[]{0.20\textwidth}
  \centering
  \begin{tikzpicture}[scale=0.8,
    coat/.style={rectangle, minimum size=2mm,draw=black},
      conn/.style={circle,inner sep=0mm,minimum size=2mm,draw=black},
      lone/.style={circle,inner sep=0mm,minimum size=1mm,draw=black,fill=black}]
    \node[coat, label=above:$a$] (a) at (0,1) {};
    \node[coat, label=above:$b$] (b) at (1,1) {};
    \node[coat, label=above:$c$] (c) at (2,1) {};
    \node[coat, label=above:$d$] (d) at (3,1) {};
    \node[conn] (p) at (0.5,0) {};
    \node[conn] (q) at (1.5,0) {};
    \node[conn] (r) at (2.5,0) {};
    \draw[-] (a) -- (p);
    \draw[-] (b) -- (p);
    \draw[-] (b) -- (q);
    \draw[-] (c) -- (q);
    \draw[-] (c) -- (r);
    \draw[-] (d) -- (r);
  \end{tikzpicture}
  \end{minipage}
  \end{center}
  \end{samepage}
  The automorphism group $G$ has two elements.  Their actions on the
  coatoms, written in the disjoint cycle notation, are the identity
  permutation $g_1 = (a)(b)(c)(d)$ and the reflection $g_2 =
  (a,d)(b,c)$.  Now $g_1$ has four cycles of length one, and $g_2$ has
  two cycles of length two, so their cycle index monomials are
  $z_{g_1}(t_1,t_2,t_3,t_4) = t_1^4$ and $z_{g_2}(t_1,t_2,t_3,t_4) =
  t_2^2$.  Thus we have
  \[
  Z_G(t_1,t_2,t_3,t_4) = \frac{1}{2} \left( t_1^4 + t_2^2 \right).
  \]
  Let us then count the ways, up to symmetry, of distributing $n$
  balls (loners) to the four boxes (coatoms), when
  $n=0,1,2,\ldots,10$.  For this finite sequence we can use a
  figure-counting polynomial $ A(x) = 1 + x + x^2 + \ldots + x^{10}. $
  Substituting this into \eqref{eq:functioncounting} we obtain a
  polynomial whose low-order coefficients, from $x^0$ to $x^{10}$,
  indicate the counts we desire: $1, 2, 6, 10, 19, 28, 44, 60, 85,
  110, 146.$ \hfill$\square$
\end{example}

Of course, the counting problem in Example~\ref{ex:simple} could have
been solved by more elementary combinatorial methods.  We may also
recognize the counts as the OEIS sequence A005993, hence finding
useful formulas and literary references.  But the benefit of employing
the cycle index theorem is that the process can be fully automated.
With a relatively simple GAP program we can easily process millions of
connection graphs.  This is illustrated in the next section.

\section{Implementation and numerical results}
\label{sec:implementation}

The first phase, generating the connection graphs, was performed for
$c=2,3,\ldots,9$ with \texttt{genbgL}, using a script similar
to~\eqref{eq:script}.  The connection graphs were stored in the
machine-readable Graph6 format~\cite{nauty}, and further compressed
with \texttt{xz}~\cite{xz}, which is quite efficient here: for $c=9$,
the Graph6 file takes $4.3$~GB before compression (52 bytes per
graph), but only $144$~MB after compression (1.7 bytes per graph).

The second phase was implemented as a GAP program.  It uses
\texttt{IteratorFromDigraphFile} from the Digraphs package to read in
the connection graphs.  For each graph it counts the ball
distributions as described in Section~\ref{sec:balls}, and by adding
up the counts it computes $R(c,a)$.  The program is here summarized in
Algorithm~\ref{alg:main}, and its full code is available in
Bitbucket~\cite{bitbucket}.

\begin{algorithm}[tb]
  \caption{Given $c, a_\text{max}$, calculate $R(c,a)$ for $a=0,1,2,\ldots,a_\text{max}$.}
  \label{alg:main}
  \begin{algorithmic}[1]
    \State $R \gets (0,\ldots,0)$
    \Comment{Initialize counts: $a_\text{max}$+1 zeros}
    \State $\mathcal{B} \gets ()$
    \Comment{Lookup table for results by cycle index}
    \For{$\ItGamma \in \texttt{IteratorFromDigraphFile}$}
    \Comment{Read in connection graphs}
      \State $c \gets \text{number of coatoms in $\ItGamma$}$
      \State $r \gets \text{number of connectors in $\ItGamma$}$
      \State $s \gets \text{number of coatoms not covering atoms in $\ItGamma$}$
      \State $G \gets \texttt{AutomorphismGroup}(\ItGamma)$
      \Comment{In Digraphs package}
      \State $Z_G \gets \texttt{CycleIndex}(G, \{1,\ldots,c\})$
      \Comment{GAP builtin function}
      \If{$\mathcal{B}(Z_G)$ exists}
      \Comment{Look up}
        \State $B \gets \mathcal{B}(Z_G)$
        \Comment{Use the memorized result}
      \Else
        \State $B \gets \texttt{GroupBalls}(Z_g, c, a_\text{max})$
        \Comment{Count ball distributions}
        \State $\mathcal{B}(Z_G) \gets B$
        \Comment{Memorize result}
      \EndIf
      \State Shift $B$ up by $r+s$ positions, prepending zeros, truncating after $a_\text{max}+1$.
      \State $R \gets R+B$
      \Comment{Accumulate counts}
    \EndFor
    \State \Return $R$
  \end{algorithmic}
\end{algorithm}

The values of $R(c,a)$ were computed up to $c=9$ and $a=1000$.  Some
aspects of the computation are illustrated in
Table~\ref{table:computation}.  Note that although there is a large number of
connection graphs, many graphs can have the same cycle index, and
\texttt{GroupBalls} is called only once for each different cycle index
(we use a lookup table).

The fourth column in Table~\ref{table:computation} shows the count of
connection graphs whose structure makes all coatoms distinguished,
that is, graphs such that the action of the automorphism group on the
coatoms is trivial.  Most connection graphs are of this type when
$c=8$ or $c=9$.  This is similar to the phenomenon that large finite
graphs tend to have trivial automorphism groups (cf. sequence A003400
in OEIS~\cite{oeis}).

With $c=9$, most time was spent in generating the connection graphs
and analyzing their symmetries.  In comparison, calculating the counts
with \texttt{GroupBalls} took relatively little time.  The computation
times should be taken as indicative only.  They are the elapsed times
when running on a single Intel Xeon E5-2680 core, with a nominal clock
frequency of 2.4~GHz.  The implementation was certainly not fully
optimized and there may be plenty of room for improvement.  We used
the following program versions: \texttt{genbgL} 1.4, GAP
4.8.10~\cite{gap}, and Digraphs 0.11.0~\cite{digraphs}.

Some of the numbers for $c \le 9$ and $a \le 1000$ are listed in
Tables~\ref{table:values37} and~\ref{table:values89}.  Since
$R(9,1000) \approx 1.775 \times 10^{27}$, it would not have been
practical to count the lattices by generating them.  Full listings of
the numerical results and the connection graphs are available in
Bitbucket~\cite{bitbucket}.

\begin{table}[tb]
\caption{Some details of the computations.
  Computation time is broken into parts:
  genbg = generating the connection graphs, aut+cyc
  = finding automorphism groups and cycle indices, counting =
  calculating the counts.\label{table:computation}}
{\begin{tabular}{crrrrrr}
\toprule
      & \# connection  & \# cycle   & \# graphs with   & time (s) & time (s) & time (s) \\
  $c$ & graphs         & indices    & trivial action   & genbg    & aut+cyc  & counting \\
\midrule
2 &          2 &    1 &        0 &     0.0  &      0.0 &      0.1 \\
3 &          5 &    2 &        0 &     0.0  &      0.0 &      0.7 \\
4 &         16 &    6 &        0 &     0.0  &      0.0 &      4.0 \\
5 &         72 &   11 &        2 &     0.0  &      0.1 &     13.0 \\
6 &        592 &   26 &      101 &     0.0  &      0.5 &     49.6 \\
7 &      10808 &   38 &     4716 &     0.7  &      7.7 &    122.0 \\
8 &     552251 &   87 &   400840 &    61.4  &    284.5 &    459.6 \\
9 &   82856695 &  142 & 74449534 & 19200.7  &  44211.1 &   1297.7 \\
\bottomrule
\end{tabular}}
\end{table}

\section{Obtaining functional forms}
\label{sec:functional}

In this section we show that for any $c \ge 1$, the function $R(c,a)$
has a representation as a quasipolynomial in~$a$ (an existence
result).  Furthermore we find explicit quasipolynomials when
$c \le 7$.

We use Lison\v{e}k's results~\cite{lisonek1994}, which show that the
number of distributions of $n$~balls into boxes
(Problem~\ref{prob:boxes}) is a quasipolynomial in~$n$.  In our
application, we take the sum of such quasipolynomials over all
connection graphs to obtain $R(c,a)$.

For each connection graph $\ItGamma$, let $G=\Aut(\ItGamma)$, and
consider its restricted action on the $c$~coatoms of~$\ItGamma$.
The function-counting series~\eqref{eq:functioncounting} is a
generating function whose coefficients are the numbers of ways to
distribute $n \ge 0$ balls to the $c$ coatoms.  The series has a
special form, a rational function whose denominator consists of
factors of the form $(1-x^i)$.  From this observation it follows
(Lison\v{e}k's Theorem 4.3.5~\cite{lisonek1994}) that the coefficients
are quasipolynomials of~$n$.

A function $f: \Naturals \to \Complex$ is a \emph{quasipolynomial} of
quasiperiod~$N$, if~there are polynomials $P_0,P_1,\ldots,P_{N-1}$
such that for all $n \ge 0$,
\begin{equation*}
  f(n) = P_k(n) \qquad\text{when $n \equiv k \pmod N$}.
\end{equation*}
Note that $N$ need not be minimal. This definition is from
Stanley~\cite{stanley2012}.  Lison\v{e}k uses a slightly different
definition that only requires the function to agree with the
polynomials from some point $n_0$ onwards.  We shall use Stanley's
definition here.

Quasipolynomials, also known as polynomials on residue classes (PORC),
are extremely versatile for expressing combinatorial quantities, which
often depend on the parity of the argument, or its residue modulo some
integer $N$.  Typical examples are floors of polynomials, such as
\eqref{eq:p2}---\eqref{eq:p21}.
We recall the following from Stanley:

\begin{proposition}[Part of Proposition 4.4.1~\cite{stanley2012}]
  A function $f : \Naturals \to \Complex$ is a quasipolynomial of
  quasiperiod $N$, if it has a rational generating function
  \[
  \sum_{n \ge 0} f(n)x^n = \frac{P(x)}{Q(x)},
  \]
  where $P$, $Q$ are polynomials; every root $\alpha$ of $Q$
  satisfies $\alpha^N = 1$; and $\deg P < \deg Q$.
  \label{prop:stanley}
\end{proposition}

Next we re-state Lison\v{e}k's result using Stanley's definition and
with explicit mention of a quasiperiod.  If~$G$ is a permutation group
acting on $c$ boxes, we denote by $P_G(n)$ the number of distributions
of $n \ge 0$ balls in the $c$ boxes, when distributions related by an
element of $G$ are not distinguished.

\begin{proposition}[cf. Theorem 4.3.5~\cite{lisonek1994}]
  The function $n \mapsto P_G(n)$ is a quasipolynomial of quasiperiod
  $N=\lcm(1,2,\ldots,c)$.
  \label{prop:lisonek}
\end{proposition}

\begin{proof}
  The generating function of $P_G(n)$ is $B(x)$ as defined
  in~\eqref{eq:functioncounting}.  Now $Z_G$ is a polynomial in $c$
  indeterminates $t_1,\ldots,t_c$, and $B(x)$ is obtained from $Z_G$
  by the P\'olya substitutions $t_i = 1/(1-x^i)$ for $i=1,2,\ldots,c$.
  With this substitution, each cycle index monomial $z_g$ becomes a
  rational function $R_g(x) = 1/Q_g(x)$, where (1)~the denominator is
  a nonempty product of binomials of the form $(1-x^i)$, and (2)~the
  degree of the numerator is strictly smaller than the degree of the
  denominator.  Clearly both conditions are retained when all $R_g(x)$
  are expanded to have a common denominator; call it $Q(x)$.

  Taking the sum of the expanded rational functions and dividing by
  the constant $|G|$, we obtain $B(x)$ as a rational function
  $P(x)/Q(x)$ with $\deg P < \deg Q$, and $Q(x)$ consisting of a
  product of binomials of the form $(1-x^i)$, with $i \in
  \{1,2,\ldots,c\}$.

  Let $N = \lcm(1,2,\ldots,c)$.  Every root $\alpha$ of $Q$ is a root
  of $(1-x^i)$ for some $i \in \{1,2,\ldots,c\}$, thus $\alpha^i = 1$
  and also $\alpha^N = 1$.

  The conditions of Proposition~\ref{prop:stanley} are satisfied, so
  the claim follows.
\end{proof}

A similar result follows for $R(c,a)$, the number of rank-3 lattices,
seen as a function of $a$ with $c$ fixed.  But we have to be careful
with the initial terms of the sequence.  We say that two functions $f$
and $g$ \emph{agree from $n_0$} if $f(n)=g(n)$ for all $n \ge n_0$.
If $\ItGamma$ is a connection graph and $a\ge 0$, we write
$R(\ItGamma,a)$ for the number of rank-3 lattices with the connection
graph $\ItGamma$ and $a$ atoms in total.

\begin{proposition}
  For any fixed $c \ge 1$, the function $a \mapsto R(c,a)$ agrees from
  $\binom{c}{2}$ with a quasipolynomial in~$a$, of quasiperiod
  $N=\lcm(1,2,\ldots,c)$.
  \label{prop:quasi}
\end{proposition}

\begin{proof}
  Let $N = \lcm(1,2,\ldots,c)$ and $n_0 = \binom{c}{2}$.

  Let $\ItGamma$ be any connection graph of $c$ coatoms, and let it
  have $r$ connectors and $s$ coatoms lacking a connector.  It is
  clear that $r+s \le n_0$.  For $a < r+s$, we have $R(\ItGamma,a) =
  0$, and for $a \ge r+s$ we have $R(\ItGamma,a) = P_G(a-r-s)$.
  Clearly $R(\ItGamma,a)$ agrees with a quasipolynomial of quasiperiod
  $N$ from $r+s$, thus also from $n_0 \ge r+s$.

  Let then $\mathcal{G}$ be the set of all connection graphs of
  $c$~coatoms.  Then
  \[
  R(c,a) = \sum_{\ItGamma \in \mathcal{G}} R(\ItGamma,a).
  \]
  Since every $R(\ItGamma,a)$ agrees with a quasipolynomial of
  quasiperiod $N$ from $n_0$, so does their sum, as the family of
  quasipolynomials is closed under finite addition.
\end{proof}

This (like Lison\v{e}k's Theorem 4.3.5) is an existence result, and
does not tell us what the quasipolynomial is.  There are many ways how
one may seek the exact form.  One could inspect the generating
function itself, perhaps prove some recurrence relations, and so on.
Here we take the low road: for each residue class $k$, if we know that
the constituent polynomial $P_k$ has degree at most~$d$, fit a
polynomial of degree~$d$ to at least $d+1$ known values.  If done in
rational (not floating-point) arithmetic, this identifies the
polynomial coefficients exactly.  But first we need to have an upper
bound on~$d$, so that we know how many points we need.  Fortunately we
can do this by an elementary argument.

  \begin{samepage}
\begin{proposition}
  For any $c \ge 1$, there is a constant $K$ such that $R(c,a) \le
  Ka^{c-1}$ for all $a \ge 1$.
  \label{prop:degree}
\end{proposition}
  \end{samepage}

\begin{proof}
  Let $\ItGamma$ be any connection graph with $c$ coatoms.  Suppose
  first that $\Aut(\ItGamma)$ fixes all coatoms, that is, we have
  $c$~distinguished boxes.  At least one atom is used by the
  connection graph, so the number of balls to distribute is $n \le
  a-1$.  There are at most $a^{c-1}$ ways to put $a-1$ balls to $c$
  distinguished boxes, since the numbers of balls in the first $c-1$
  boxes are a $(c-1)$-tuple of integers between $0$ and $a-1$.  Then
  observe that if $\Aut(\ItGamma)$ does not fix all coatoms, this can
  only decrease the number of distributions.  So for all connection
  graphs we have $R(\ItGamma,a) \le a^{c-1}$.  For a fixed value of
  $c$, there is a finite collection $\mathcal{G}$ of connection
  graphs, so $R(c,a) \le |\mathcal{G}| \; a^{c-1}$.  Take $K =
  |\mathcal{G}|$.
\end{proof}

Now, by Propositions~\ref{prop:quasi} and~\ref{prop:degree}, we know
the following:
\begin{itemize}
\item $R(4,a)$ agrees with a quasipolynomial with $N=\lcm(1,2,3,4)=12$
  and degree at most~$3$, from $n_0=\binom{4}{2}=6$.  For the fit we
  need $12 \times 4 = 48$ known values.
\item $R(5,a)$ agrees with a quasipolynomial with
  $N=\lcm(1,\ldots,5)=60$ and degree at most~$4$, from
  $n_0=\binom{5}{2}=10$.  For the fit we need $60 \times 5 = 300$
  known values.
\item $R(6,a)$ agrees with a quasipolynomial with
  $N=\lcm(1,\ldots,6)=60$ and degree at most~$5$, from
  $n_0=\binom{6}{2}=15$.  For the fit we need $60 \times 6 = 360$
  known values.
\item $R(7,a)$ agrees with a quasipolynomial with
  $N=\lcm(1,\ldots,7)=420$ and degree at most~$6$, from
  $n_0=\binom{7}{2}=21$.  For the fit we need $420 \times 7 = 2940$
  known values.
\end{itemize}

The values of $n_0$ are as guaranteed by Proposition~\ref{prop:quasi}.
Once a quasipolynomial is available, one can check if it happens to
agree with some of the initial terms, and extend its range
accordingly.

As we already computed the values up to $a=1000$ before, we have
enough values to fit the polynomials up to $c=6$.  For $c=7$ we reran
Algorithm~\ref{alg:main} with $a_\text{max}=3000$, which took about
$1100$ seconds.

From the polynomial fits we obtain the following quasipolynomials.
Leading terms that are common to all residue classes are collected
together.  For the terms that depend on residue class, we use this
shorthand notation: the quantity $[c_0,c_1,\ldots,c_{M-1}]$ takes the
value $c_k$ when $a \equiv k \pmod M$.  For example, $[3,-3]$ means
$3$ if $a$ is even, and $-3$ if $a$ is odd.  Note that $M$ may be
smaller than~$N$.

\begin{theorem}
  For any $a\ge 0$, we have
  \begin{multline*}
    R(4,a) = (97/144)a^3 - (5/6)a^2 
    + [44/48,47/48] \, a \\
    + [0,13,8,-45,40,-19,0,-5,8,-27,40,-37] \; / \; 72.
  \end{multline*}
  \label{thm:r4}
\end{theorem}

\begin{theorem}
  For any $a\ge 3$, we have
  \begin{multline*}
    R(5,a) = \; (175/192)a^4 - (3079/480)a^3 + (11771/480)a^2 \\
    + [-7268/160, -7273/160] \, a \; + h(a),
  \end{multline*}
  where
  \begin{multline*}
    h(a) = [
      33600, 34019, 34072, 33627, 33152, 34915, 33624, 33947, 33472, 33507,\\
      34520, 34459, 32832, 33827, 34072, 34395, 33344, 34147, 33432, 33947,\\
      34240, 33699, 33752, 34267, 32832, 34595, 34264, 33627, 33152, 34147,\\
      34200, 34139, 33472, 33507, 33752, 35035, 33024, 33827, 34072, 33627,\\
      33920, 34339, 33432, 33947, 33472, 34275, 33944, 34267, 32832, 33827,\\
      34840, 33819, 33152, 34147, 33432, 34715, 33664, 33507, 33752, 34267 ] \; / \; 960.
  \end{multline*}
  \label{thm:r5}
\end{theorem}

It was indeed necessary to be careful with initial terms.  The
quasipolynomial in Theorem~\ref{thm:r5} does \emph{not} agree with
$R(5,a)$ at $a=0,1,2$.  At those points, the quasipolynomial yields
$35,9,6$, while the true values of $R(5,a)$ are $0,1,5$.

There does not seem to be much structure in the ``quasiconstant''
$h(a)$ term of $R(5,a)$.  It has the full period $60$.  Perhaps this
was to be expected: from Table~\ref{table:computation} we recall that
$R(5,a)$ ensues as a sum over $72$ graphs having $11$ different cycle
indices.  If the graphs have residue-dependent contributions with
different periods, their joint effect will easily end up with the full
period.  If desired, one may try to manipulate $h(a)$ into a form that
is more pleasant to human eyes.
  
It would have been quite tedious to derive $R(5,a)$ manually, even if
to find just the leading terms.  By the help of computations we can
readily see the three leading terms, which do not depend on residue
class, and which provide a fairly precise picture of the growth rate
of $R(5,a)$.  And if we want the fine details, they are there.  With
the explicit formula one can easily calculate, say,
$R(5,1000000) = 911451918774522871241702$.

For reasons of clarity, for $c=6$ and $c=7$ we show here just the
leading terms that are common to all residue classes, and hide the
lower order terms behind an $O$~notation.  The full explicit
quasipolynomials are available in Bitbucket~\cite{bitbucket}.

\begin{theorem}
  \[
  R(6,a) = (185521/86400)a^5 - (266581/6912)a^4 + (4268287/12960)a^3 + O(a^2).
  \]
  \label{thm:r6}
\end{theorem}

\begin{theorem}
  \begin{multline*}
  R(7,a) = (35406319/3628800)a^6 - (205303771/604800)a^5 \\
  +(986460817/181440)a^4 - (908874965/18144)a^3 + O(a^2).
  \end{multline*}
  \label{thm:r7}
\end{theorem}

As a redundancy check, for all available values of $a$, we compared
the quantities evaluated from the quasipolynomials against the values
directly obtained from Algorithm~\ref{alg:main}, and observed that
they agree.

Finally, let us see what the computations have to say about simple
cases.  By the same polynomial-fit method as above, for all $a \ge 0$
we obtain
\begin{align*}
  R(2,a) &= a, \\
  R(3,a) &= (3/4)a^2+ (1/3)a +  [ 0, -1, -8, 3, -4, -5 ] \; / \;12,
\end{align*}
giving an alternative derivation of our first basic results in
Section~\ref{sec:three}.

\section{Closing remarks}
\label{sec:closing}

The computations in this work were facilitated by the availability of
several tools.  For the first phase, where the connection graphs were
generated, a tool of isomorph-free graph generation (\texttt{genbgL}
from the nauty package) was essential.  The second phase required the
mathematical machinery of counting theory; but for actually computing
the automorphism groups and cycle indices of millions of graphs, it
was convenient that those tools were available in GAP and Digraphs.
Easy-to-use arithmetic on large integers, rationals, and polynomials
was also helpful.

The method of converting the generating functions to functional forms
by computing initial terms and then doing a polynomial fit has a
certain ``snake oil'' appeal.  It is easy to do, provided that one has
a suitable existence result that the quasipolynomial actually is
there to be found.  However, it may require rather long sequences to
be computed.  Perhaps a computational method that inspects the
structure of the generating function directly would be more efficient
here.

Here we considered only one kind of lattices, the rank-3 lattices.
Some of the methods used here may be applicable to other low-rank
graded lattices as well.  A natural next question would be that of
rank-4 lattices.  What is $R(c,m,a)$, the number of graded lattices of
$c$~coatoms, $m$~elements in the middle level, and $a$~atoms?  Can one
simply ``glue'' two rank-3 lattices on top of each other?  Perhaps,
but with two considerations: that of isomorphism, and that of ensuring
that the glued results are indeed lattices.

Another thing for further study would be to implement an
\emph{enumerator}, that is, a computational function that creates a
rank-3 lattice on demand, when given an indexing integer $i$ such that
$1 \le i \le R(c,a)$.


\bibliographystyle{plain}
\bibliography{refs}


\begin{table}[p]
  \caption{Some values of $R(c,a)$ (see \protect\cite{bitbucket} for full results).
    \label{table:values37}}
{\begin{tabular}{rrrrrr}
\toprule    
 $a$ & $R(3,a)$ & $R(4,a)$ & $R(5,a)$ & $R(6,a)$ & $R(7,a)$\\
\midrule
   1  & 1 & 1 & 1 & 1 & 1  \\
   2  & 3 & 4 & 5 & 6 & 7  \\
   3  & 8 & 13 & 20 & 29 & 39  \\
   4  & 13 & 34 & 68 & 121 & 197  \\
   5  & 20 & 68 & 190 & 441 & 907  \\
   6  & 29 & 121 & 441 & 1384 & 3736  \\
   7  & 39 & 197 & 907 & 3736 & 13530  \\
   8  & 50 & 299 & 1690 & 8934 & 42931  \\
   9  & 64 & 432 & 2916 & 19298 & 120892  \\
  10  & 78 & 600 & 4734 & 38268 & 306120  \\
  11  & 94 & 806 & 7310 & 70685 & 706642  \\
  12  & 112 & 1055 & 10836 & 123057 & 1506016  \\
  13  & 131 & 1352 & 15528 & 203764 & 2996398  \\
  14  & 151 & 1698 & 21619 & 323383 & 5618515  \\
  15  & 174 & 2100 & 29365 & 494925 & 10008899  \\
  16  & 197 & 2561 & 39045 & 734034 & 17053898  \\
  17  & 222 & 3085 & 50961 & 1059330 & 27950691  \\
  18  & 249 & 3675 & 65434 & 1492653 & 44275741  \\
  19  & 277 & 4338 & 82809 & 2059229 & 68059684  \\
  20  & 306 & 5074 & 103453 & 2788044 & 101869637  \\
  21  & 338 & 5891 & 127751 & 3712081 & 148898469  \\
  22  & 370 & 6790 & 156117 & 4868468 & 213061109  \\
  23  & 404 & 7777 & 188980 & 6298878 & 299097442  \\
  24  & 440 & 8854 & 226794 & 8049751 & 412683316  \\
  25  & 477 & 10029 & 270037 & 10172443 & 560547117  \\
  26  & 515 & 11300 & 319204 & 12723627 & 750594650  \\
  27  & 556 & 12677 & 374813 & 15765529 & 992040210  \\
  28  & 597 & 14160 & 437409 & 19366035 & 1295545409  \\
  29  & 640 & 15756 & 507553 & 23599151 & 1673363704  \\
  30  & 685 & 17465 & 585831 & 28545198 & 2139494240  \\
 \ldots \\
 100  & 7533 & 665370 & 84971972 & 17929736129 & 6858729229937  \\
 200  & 30066 & 5355739 & 1407988534 & 627979574932 & 524132826147936  \\
 300  & 67600 & 18112775 & 7211812220 & 4914131994972 & 6330705903535897  \\
 400  & 120133 & 42978145 & 22926705532 & 21021167741959 & 36624782962133435  \\
 500  & 187666 & 83993514 & 56170430969 & 64731346381612 & 142179199873933941  \\
 600  & 270200 & 145200550 & 116748251030 & 162041086855752 & 429521796157985802  \\
 700  & 367733 & 230640920 & 216652928217 & 351737648034289 & 1092140851049830127  \\
 800  & 480266 & 344356289 & 370064725029 & 687975809274792 & 2448715582864593496  \\
 900  & 607800 & 490388325 & 593351403965 & 1242854550978032 & 4988371711653746757  \\
1000  & 750333 & 672778695 & 905068227527 & 2108993735138119 & 9422962085155489652  \\
\bottomrule
\end{tabular}}
\end{table}

\begin{table}[p]
  \caption{Some values of $R(c,a)$ (see \protect\cite{bitbucket} for full results).
    \label{table:values89}}
{\begin{tabular}{rrr}
\toprule
  $a$ & $R(8,a)$ & $R(9,a)$\\
\midrule
   1  & 1 & 1  \\
   2  & 8 & 9  \\
   3  & 50 & 64  \\
   4  & 299 & 432  \\
   5  & 1690 & 2916  \\
   6  & 8934 & 19298  \\
   7  & 42931 & 120892  \\
   8  & 183303 & 690896  \\
   9  & 690896 & 3517049  \\
  10  & 2310366 & 15818049  \\
  11  & 6920971 & 63028260  \\
  12  & 18783412 & 224257964  \\
  13  & 46705657 & 719521493  \\
  14  & 107510169 & 2102741467  \\
  15  & 231227596 & 5650968147  \\
  16  & 468463678 & 14088437189  \\
  17  & 900399211 & 32842695865  \\
  18  & 1651885113 & 72096705250  \\
  19  & 2908101609 & 149972933224  \\
  20  & 4935241680 & 297260914919  \\
  21  & 8105691264 & 564176756133  \\
  22  & 12928165761 & 1029721046925  \\
  23  & 20083274851 & 1814279741924  \\
  24  & 30464974385 & 3096191012173  \\
  25  & 45228381098 & 5133079209599  \\
  26  & 65844403276 & 8288835750730  \\
  27  & 94161667324 & 13067204701747  \\
  28  & 132476193092 & 20153009591032  \\
  29  & 183609295480 & 30462135974619  \\
  30  & 250994166078 & 45201463018088  \\
 \ldots \\
 100  & 5078592962561811 & 7626564586350129874  \\
 200  & 880085483053191106 & 3142649707966986066096  \\
 300  & 16609587584876364182 & 94045317769328410172825  \\
 400  & 130737521692628355615 & 1014377064737641167135036  \\
 500  & 642112898798336927353 & 6329853496024443260170625  \\
 600  & 2346516577212608845729 & 28059449401711567076441545  \\
 700  & 7000760472426076825846 & 98420943238637719981239097  \\
 800  & 18015850571650533933600 & 291130542533101026907632456  \\
 900  & 41425805120978743606026 & 756477905666369353284138046  \\
1000  & 87178719353101913391613 & 1775181449515604936706800068 \\
\bottomrule
\end{tabular}}
\end{table}

\end{document}